\documentclass{article}
\usepackage{amssymb}
\usepackage{amsmath}
\usepackage{amsfonts}
\usepackage[linkcolor=blue,colorlinks=true,urlcolor=blue,citecolor=blue]{hyperref}
\usepackage{color,rotating}

\newtheorem{theorem}{Theorem}

\newtheorem{proposition}[theorem]{Proposition}
\newtheorem{remark}[theorem]{Remark}

\newenvironment{proof}[1][Proof]{\noindent\textbf{#1.} }{\ \rule{0.5em}{0.5em}}

\begin{document}

\title{Note on a Lyapunov-type inequality for a fractional boundary value
problem with Caputo-Fabrizio derivative }
\author{Zaid Laadjal \\ \\
Department of Mathematics, ICOSI Laboratory,\\
University of Abbes Laghrour, Khenchela, 40000, Algeria.\\
E-mail: zaid.laadjal@yahoo.com }

\date{}
\maketitle

\begin{abstract}
In this short note, we present a Lyapunov-type inequality that corrects the 
recently obtained result in [M. Kirane, B. T. Torebek: A Lyapunov-type 
inequality for a fractional boundary value problem with Caputo-Fabrizio 
derivative, J. Math. Inequal. 12, 4 (2018), 1005--1012].
\end{abstract}



\section{Introduction}

Recently, in \cite{Kirane} the authors discussed a Lyapunov-type inequality
for the following linear fractional boundary value problem:%
\begin{equation}
\left\{ 
\begin{array}{cc}
^{CF}D_{a}^{\alpha }u(t)+q(t)u(t)=0, & 0\leq a<t<b, \\ 
u(a)=u(b)=0, & 
\end{array}%
\right.  \label{1}
\end{equation}%
where $^{CF}D_a^{\alpha }$ denotes the Caputo-Fabrizio derivative \cite%
{Fabrizio,Losada} of order $\alpha,(1<\alpha \leq 2),$ $q:[a,b]\rightarrow 
\mathbb{R}
$ is a continuous function$.$ And they included the following result:

\begin{theorem}[\protect\cite{Kirane}]
\label{Th1} If the fractional boundary value problem (\ref{1}) has a 
nontrivial solution, then  
\begin{equation}
\int_{a}^{b}|q(t)|ds>\frac{4(\alpha -1)(b-a)}{[(\alpha -1)(b-a)-2+\alpha
]^{2}}.  \label{5}
\end{equation}
\end{theorem}

We\ have noticed that, the denominator in the inequality (\ref{5}) is equal
zero when $b-a=\frac{2-\alpha }{\alpha -1}.$ So a mistake has been occured
during the previous result and some other results (Corrolary 3.4 and
Corrolary 3.5 in \cite{Kirane}) are also incorrect. These mistakes come from
the main wrong in (Lemma 3.2 in \cite{Kirane} related to the calculations
for the maximum value of the Green's function of the problem (\ref{1}).

This work aims to show these mistakes and present the correct version of
them.

We will also refer the interested reader in studying the Lyapunop-type
inequalities for a fractional boundary value problems to the works compiled
in chapter in the book \cite{Ntouyas}, as well as some other papers
published recently, for example see \cite{Jleli1}--\cite{Jleli3} and the
references cited therein.

\section{Main results}

\label{Sec:2}

The fractional boundary value problem (\ref{1}) is equivalent to the
integral equation 
\begin{equation}
u(s)=\int_{a}^{b}G(t,s)q(s)u(s)ds,  \label{int_eq}
\end{equation}
where $G(t,s)$ is called the Green's function of the problem (\ref{1}) and
it's difened by%
\begin{equation}
G(t,s)=\left\{ 
\begin{array}{lll}
g_{1}(t,s)=\frac{b-t}{b-a}[(\alpha -1)(s-a)-2+\alpha ], &  & a\leq s\leq
t\leq b, \\ 
&  &  \\ 
g_{2}(t,s)=\frac{t-a}{b-a}[(\alpha -1)(b-s)+2-\alpha ], &  & a\leq t\leq
s\leq b.%
\end{array}%
\right.  \label{3}
\end{equation}%
See \cite{Kirane} for more details, (note that this function in (Lemma 3.1, 
\cite{Kirane}) is written in wrong way. Although its proof is true).

The mistake alluded to in (\cite{Kirane}, Section 3) is that the authors
concluded that the maximum value of the function $G(t,s)$ is obtained at the
point 
\begin{equation}
t=s=\frac{1}{2}\left( b+a+\frac{2-\alpha }{\alpha -1}\right) :=s^{\ast },
\label{s*}
\end{equation}%
\ where $s^{\ast }$\ is defined by (equality (3.5) in \cite{Kirane}).
However, this is wrong for $(t,s)\in \lbrack a,b]\times \lbrack a,b]$ with $%
1<\alpha \leq 2$, as we will show nextly.

Let us start to discuss the previous value of $s^{\ast }$.\newline
Note that, if $b-a<\frac{2-\alpha }{\alpha -1},$ we have%
\begin{eqnarray*}
&& 
\begin{array}{lll}
b-a<\frac{2-\alpha }{\alpha -1} & \Longleftrightarrow & 2b<b+a+\frac{
2-\alpha }{\alpha -1}%
\end{array}
\\
&& 
\begin{array}{lll}
\text{ \ \ \ \ \ \ \ \ \ \ \ \ \ \ \ \ \ \ \ } & \Longleftrightarrow & b< 
\frac{1}{2}\left( b+a+\frac{2-\alpha }{\alpha -1}\right)%
\end{array}
\\
&& 
\begin{array}{lll}
\text{ \ \ \ \ \ \ \ \ \ \ \ \ \ \ \ \ \ \ \ } & \Longleftrightarrow & 
b<s^{\ast },%
\end{array}%
\end{eqnarray*}%
thus $s^{\ast }\notin \lbrack a,b].$ Then the maximum value of the function $%
G(t,s)$ is not at $s^{\ast }$ when $b-a<$ $\frac{2-\alpha }{\alpha -1}.$%
\newline
Now, for $a\leq s\leq t\leq b,$ with $b-a<\frac{2-\alpha }{\alpha -1},$\ we
have 
\begin{eqnarray}
&&%
\begin{array}{ccc}
b-a<\frac{2-\alpha }{\alpha -1} & \Longleftrightarrow & b<a+\frac{2-\alpha }{
\alpha -1}%
\end{array}
\notag \\
&&%
\begin{array}{ccc}
\text{ \ \ \ \ \ \ \ \ \ \ \ \ \ \ \ \ \ \ \ } & \Longrightarrow & s<a+\frac{
2-\alpha }{\alpha -1}%
\end{array}
\notag \\
&&%
\begin{array}{ccc}
\text{ \ \ \ \ \ \ \ \ \ \ \ \ \ \ \ \ \ \ \ } & \Longrightarrow & (\alpha
-1)(s-a)-2+\alpha <0,%
\end{array}
\label{6.a}
\end{eqnarray}%
on other hand, we have 
\begin{equation}
\frac{b-t}{b-a}\leq \frac{b-s}{b-a}.  \label{7}
\end{equation}%
By the inequalities (\ref{6.a}) and (\ref{7}) we get 
\begin{equation}
g_{1}(t,s)\geq \frac{b-s}{b-a}[(\alpha -1)(s-a)-2+\alpha ],\text{ \ }a\leq
s\leq b<a+\frac{2-\alpha }{\alpha -1}.  \label{9}
\end{equation}%
Observe that the inequality (\ref{9}) is contrary to (the inequality (3.4)
in \cite{Kirane}).

\begin{remark}
Note that on a general interval $[a,b],0\leq a<b,$ we have  
\begin{equation}
h_{1}(s)\leq g_{1}(t,s)\leq 0,\text{ for\ }a\leq s\leq t\leq a+\frac{
2-\alpha }{\alpha -1},  \label{10}
\end{equation}
where the function $h_{1}$ is defined by  
\begin{equation}
h_{1}(s)=g_{1}(s,s)=\frac{b-s}{b-a}[(\alpha -1)(s-a)-2+\alpha ],\text{ \ }
s\in \lbrack a,b]  \label{h1}
\end{equation}
\end{remark}

We differentiate the function $h_{1}(s)$ to get 
\begin{equation}
h_{1}^{\prime }(s)=-\frac{2(\alpha -1)}{b-a}s+\frac{(\alpha -1)\left(
b+a\right) +2-\alpha }{b-a}.  \label{16}
\end{equation}%
We have $s^{\ast }$ is the unique solution of the equation $h_{1}^{\prime
}(s)=0,$ where $s^{\ast }$\ is given by (\ref{s*}) but the value of $s^{\ast
}$\ in some cases does not belong to the interval $[a,b]$ as we have shown
previously.

By the discussion above, we can conclude that the maximum value of the
function $G(t,s)$\ lays in the following two cases:

\textbf{Case 1. }$b-a<$\textbf{\ }$\frac{2-\alpha }{\alpha -1}.$

Because $b-a<$ $\frac{2-\alpha }{\alpha -1},$ so then $s\leq b<s^{\ast }$
(here $a+\frac{2-\alpha }{\alpha -1},s^{\ast }\notin \lbrack a,b])$, we
obtain $h_{1}^{\prime }(s)\geq $ $0$ and $h_{1}(s)\leq 0$ for $s\leq b<a+%
\frac{2-\alpha }{\alpha -1}.$ So by (\ref{10}) and the contunuity of the
function $h_{1}$ we conclude that 
\begin{eqnarray}
\max_{a\leq s\leq t\leq b<a+\frac{2-\alpha }{\alpha -1}}\left\vert
g_{1}(t,s)\right\vert &=&\max_{a\leq s\leq b<a+\frac{2-\alpha }{\alpha -1}%
}\left\vert h_{1}(s)\right\vert  \notag \\
&=&-h_{1}(a)  \notag \\
&=&2-\alpha .  \label{21}
\end{eqnarray}

Next, for $a\leq t\leq s\leq b.$ Obviously, 
\begin{equation}
0\leq g_{2}(t,s)\leq \frac{s-a}{b-a}[(\alpha -1)(b-s)+2-\alpha].\text{ }
\label{12}
\end{equation}%
We define a function $h_{2}$ by%
\begin{equation}
h_{2}(s)=g_{2}(s,s)=\frac{s-a}{b-a}[(\alpha -1)(b-s)+2-\alpha ],\ a\leq
s\leq b.  \label{h2}
\end{equation}%
Differentiating the function $h_{2}(s)$%
\begin{equation}
h_{2}^{\prime }(s)=-\frac{2(\alpha -1)}{b-a}s+\frac{(\alpha -1)\left(
a+b\right) +\left( 2-\alpha \right) }{b-a},  \label{13}
\end{equation}%
which implies that the function $h_{2}^{\prime }(s)$ has a unique zero, at
the point $s^{\ast },$ but $s^{\ast }>b$ (i.e. $s^{\ast }\notin \lbrack a,b]$%
)$.$ Because $h_{2}^{\prime }(s)>0$ for all $s\in \lbrack a,b]$ and $%
g_{2}(t,s)\geq 0$ for $a\leq t\leq s\leq b,$ and $h_{2}(s)$ is continuous
function, then 
\begin{eqnarray}
\underset{a\leq t\leq s\leq b<a+\frac{2-\alpha }{\alpha -1}}{\max }%
\left\vert g_{2}(t,s)\right\vert &=&\max_{a\leq s\leq b<a+\frac{2-\alpha }{
\alpha -1}}h_{2}(s)  \notag \\
&=&h_{2}(b)  \notag \\
&=&2-\alpha .  \label{15}
\end{eqnarray}%
By (\ref{15}) and (\ref{21}) we get 
\begin{equation}
\underset{a\leq t,s\leq b<a+\frac{2-\alpha }{\alpha -1}}{\max }\left\vert
G(t,s)\right\vert =2-\alpha .  \label{22}
\end{equation}%

\textbf{Case 2.} $b-a\geq $ $\frac{2-\alpha }{\alpha -1}.$

From the inequality $b-a\geq $ $\frac{2-\alpha }{\alpha -1}$ we get $a+\frac{
2-\alpha }{\alpha -1}\in \lbrack a,b]$ and $s^{\ast }\in \lbrack a+\frac{
2-\alpha }{\alpha -1},b],$ we obtain%
\begin{equation}
\left\{ 
\begin{array}{l}
0\leq g_{1}(t,s)\leq h_{1}(s),\text{ \ }a+\frac{2-\alpha }{\alpha -1}\leq
s\leq b. \\ 
\\ 
h_{1}(s)\leq g_{1}(t,s)\leq 0,\text{ \ }a\leq s\leq a+\frac{2-\alpha }{
\alpha -1}<s^{\ast }.%
\end{array}%
\right.  \label{22.1}
\end{equation}%
Because $h_{1}(s)$ is continuous function, and $h_{1}(a+\frac{2-\alpha }{
\alpha -1})=h_{1}(b)=0,$ and $h_{1}(a)=-(2-\alpha ),$ we conclude that 
\begin{eqnarray}
\max_{a\leq s\leq t\leq a+\frac{2-\alpha }{\alpha -1}\leq b}\left\vert
g_{1}(t,s)\right\vert &=&\max_{a\leq s\leq a+\frac{2-\alpha }{\alpha -1}\leq
b}\left\vert h_{1}(s)\right\vert  \notag \\
&=&\max \left\{ -h_{1}(a),h_{1}(s^{\ast })\right\}  \notag \\
&=&\max \left\{ 2-\alpha ,\frac{[(\alpha -1)(b-a)-2+\alpha ]^{2}}{4(\alpha
-1)(b-a)}\right\} .  \label{23}
\end{eqnarray}%
On other hand, we have%
\begin{eqnarray}
\underset{a\leq t\leq s\leq a+\frac{2-\alpha }{\alpha -1}\leq b}{\max }%
\left\vert g_{2}(t,s)\right\vert &=&\max_{a\leq s\leq a+\frac{2-\alpha }{
\alpha -1}\leq b}h_{2}(s)  \notag \\
&=&h_{2}(s^{\ast })  \notag \\
&=&\frac{\left[ (\alpha -1)\left( b-a\right) +\left( 2-\alpha \right) \right]
^{2}}{4(\alpha -1)\left( b-a\right) }.  \label{23.2}
\end{eqnarray}%
By (\ref{23}) and (\ref{23.2}) we get 
\begin{eqnarray*}
\max_{a\leq t,s\leq a+\frac{2-\alpha }{\alpha -1}\leq b}\left\vert
G(t,s)\right\vert &=&\max \left\{ 2-\alpha ,\frac{[(\alpha -1)(b-a)-\left(
2-\alpha \right) ]^{2}}{4(\alpha -1)(b-a)},\right. \\
&&\text{\ \ \ \ \ \ \ \ \ \ \ \ }\left. \frac{[(\alpha -1)(b-a)+\left(
2-\alpha \right) ]^{2}}{4(\alpha -1)(b-a)}\right\} \\
&=&\max \left\{ 2-\alpha ,\frac{[(\alpha -1)(b-a)+\left( 2-\alpha \right)
]^{2}}{4(\alpha -1)(b-a)}\right\} ,
\end{eqnarray*}%
using the inequality $\frac{(A+B)^{2}}{4}\geq AB,$ with $A=(\alpha -1)(b-a)$
and $B=\left( 2-\alpha \right) ,$ we obtain 
\begin{equation}
\max_{a\leq t,s\leq a+\frac{2-\alpha }{\alpha -1}\leq b}\left\vert
G(t,s)\right\vert =\frac{[(\alpha -1)(b-a)+\left( 2-\alpha \right) ]^{2}}{
4(\alpha -1)(b-a)}.  \label{24.2}
\end{equation}

Thus we conclude the follwing result.

\begin{proposition}
\label{Prop1} The Green's function $G$ defined by (\ref{3}), has the 
following properties: \newline
$i\mathbf{).}$ If $b-a<$ $\frac{2-\alpha }{\alpha -1},$ then  
\begin{equation}
\underset{(t,s)\in \lbrack a,b]\times \lbrack a,b]}{\max }\left\vert
G(t,s)\right\vert =2-\alpha ,  \label{25}
\end{equation}
$ii).$ If $b-a\geq $ $\frac{2-\alpha }{\alpha -1},$ then  
\begin{equation}
\max_{(t,s)\in \lbrack a,b]\times \lbrack a,b]}\left\vert G(t,s)\right\vert
= \frac{[(\alpha -1)(b-a)+\left( 2-\alpha \right) ]^{2}}{4(\alpha -1)(b-a)}.
\label{26}
\end{equation}
\end{proposition}

Hence we have the following Lyapunov-type inequality.

\begin{theorem}
\label{Th2} If the fractional boundary value problem (\ref{1}) has a 
nontrivial solution. Then  
\begin{equation}
\int_{a}^{b}|q(t)|ds\geq \left\{ 
\begin{array}{l}
\frac{1}{2-\alpha .},\text{ \ if \ }b-a<\frac{2-\alpha }{\alpha -1}, \\ 
\\ 
\frac{4(\alpha -1)(b-a)}{[(\alpha -1)(b-a)+\left( 2-\alpha \right) ]^{2}}, 
\text{ \ if \ }b-a\geq \frac{2-\alpha }{\alpha -1}.%
\end{array}
\right.  \label{27}
\end{equation}
\end{theorem}

\begin{proof}
Since the proof is well-known so that the reader can easily check it on, 
where it's used in \cite{Kirane} but in here we get into details in two 
cases related the properties (\ref{25}) and (\ref{26}).
\end{proof}

By using Theorem \ref{Th2}, the reader can smoothly correct (Corrolary 3.4
and Corrolary 3.5 in \cite{Kirane}), but should seperate each of them in two
cases $b-a<$ $\frac{2-\alpha }{\alpha -1},$ and $b-a\geq $ $\frac{2-\alpha }{
\alpha -1}$.


\end{document}